\documentclass[aap,preprint]{imsart}

\RequirePackage[OT1]{fontenc}
\RequirePackage{amsthm,amsmath}
\RequirePackage[numbers]{natbib}
\RequirePackage[colorlinks,citecolor=blue,urlcolor=blue]{hyperref}

\usepackage[utf8]{inputenc}

\usepackage{graphicx}
\usepackage{subfigure}
\usepackage{float}
\usepackage{amsfonts}
\usepackage{enumitem}
\usepackage{amssymb}
\usepackage{color}
\usepackage{eucal}

\theoremstyle{plain}
\newtheorem{mythe}{Theorem}[section]
\newtheorem{mylem}[mythe]{Lemma}
\newtheorem{lemma}[mythe]{Lemma}

\newtheorem{corollary}[mythe]{Corollary}

\newtheorem{remark}[mythe]{Remark}

\newtheorem{prop}[mythe]{Proposition}

\newcommand{\ba}{\begin{eqnarray*}}
\newcommand{\ea}{\end{eqnarray*}}
\newcommand{\bq}{\begin{eqnarray}}
\newcommand{\eq}{\end{eqnarray}}

\begin{document}

\begin{frontmatter}
\title{Mutation frequencies in a birth-death branching process}
\runtitle{Mutation frequencies in a branching process}

\begin{aug}
\author{\fnms{David} \snm{Cheek}\thanksref{3}\ead[label=e1]{d.m.cheek@sms.ed.ac.uk}},
\and
\author{\fnms{Tibor} \snm{Antal}\ead[label=e2]{tibor.antal@ed.ac.uk}}

\thankstext{3}{Corresponding author. \printead{e1}}

\runauthor{D. Cheek and T. Antal}

\affiliation{University of Edinburgh}

\address{School of Mathematics, University of Edinburgh\\
Edinburgh, EH9 3FD, UK\\
\printead{e1}\\
\phantom{E-mail:\ }\printead*{e2}}

\end{aug}

\begin{abstract}
First, we revisit the stochastic Luria-Delbr{\"u}ck model: a classic two-type branching process which describes cell proliferation and mutation. We prove limit theorems and exact results for the mutation times, clone sizes, and number of mutants. Second, we extend the framework to consider mutations at multiple sites along the genome. The number of mutants in the two-type model characterises the mean site frequency spectrum in the multiple-site model. Our predictions are consistent with previously published cancer genomic data.
\end{abstract}



\end{frontmatter}

\section{Introduction}
Luria and Delbr\"{u}ck's famous work of 1943 combined mathematical modelling with experiment \cite{ld}. They considered an exponentially growing population of bacterial cells which is sensitive to attack by a lethal virus. The bacteria may mutate to become resistant to the virus. Lea and Coulson \cite{lc} obtained a probability distribution for the number of mutants, commonly known as the Luria-Delbr{\"u}ck distribution. The distribution has seen empirical evidence and become a standard tool for the estimation of mutation rates in bacteria \cite{rf}. While early formulations of the model were semi-deterministic, stochastic cell growth was subsequently incorporated (see \cite{qz} for a review). Notably, Kendall allowed for cells to grow as birth-death branching processes \cite{dk}.

Kendall's two-type branching process, often referred to as the stochastic Luria-Delbr{\"u}ck model, has been foundational in the mathematical understanding of cancer evolution. The model and various extensions have been used to study drug resistance \cite{inm,k,fl,belife}, driver mutations \cite{dm,dflmm}, and metastasis \cite{mni,hm,na,ding}, for example. As introduced by Kendall, wildtype (type A) and mutant (type B) cells are assumed to divide, die, and mutate independently of each other, according to
\ba
\begin{cases}
A\rightarrow AA,\quad&\text{rate }\alpha_A;\\
A\rightarrow \emptyset,\quad&\text{rate }\beta_A;\\
A\rightarrow AB,\quad&\text{rate }\nu;\\
B\rightarrow BB,\quad&\text{rate }\alpha_B;\\
B\rightarrow \emptyset,\quad&\text{rate }\beta_B.
\end{cases}
\ea
Whether the model represents the emergence of drug resistance in cancer or bacteria, the total number of mutants is of key interest. In recent years, \cite{wa,inm,kwb,kl1,kl,ak,hy} derived exact and approximate distributions for the number of mutants at fixed times and population sizes.

Our first objective is to offer a mathematically rigorous account of the two-type model, looking at the number of mutants, mutation times, and clone sizes (a clone is a subpopulation of mutant cells initiated by a mutation). Both previously known and new results are presented. We explore small mutation limits and long-term almost sure convergence. Specialising to neglect cell death, we give some exact distributions.

Our second objective is to introduce a neutral model of cancer evolution, which keeps track of mutations at $S$ sites on the genome. A site refers to a base pair. In our multiple-site model, each cell is labelled by a sequence $(z_1,..,z_S)\in\{0,1\}^S$, where $z_i=1$ means that the cell is mutated at site $i$. The number of mutants with respect to a particular site follows the two-type model. Thus many of the two-type results are applicable in the multiple-site setting.

A standard summary statistic of genomic data is the site frequency spectrum. It is defined as the number of sites who see mutations in $k$ cells, for $k\in\mathbb{Z}_{\geq0}$. We prove that the mean site frequency spectrum can be approximated with a generalisation of the Luria-Delbr{\"u}ck distribution. This result is consistent with cancer genomic data presented in \cite{gs,ib}.

Of course, many works have attempted to predict mutation frequencies in cancer. Prominent examples are \cite{gs,ib,oi,durr}, who gave approximations for the mean site frequency spectrum in a population of cancer cells. Every one of these works and countless others have used the infinite sites assumption, which says that each mutation occurs at a unique site. However, recent statistical analysis of cancer genomic data has refuted the validity of this simplification \cite{isr}. We do not use the infinite sites assumption, and make a theoretical argument against it.\\

The rest of the paper is organised as follows. In Section \ref{sec:mod}, we introduce the two-type model. In Section \ref{sec:lt}, we present long-term almost sure convergence results. In Section \ref{sec:lpsm}, we define and study the large population small mutation limit. In Section \ref{sec:ltsm}, we look at the large time small mutation limit. In Section \ref{sec:finite}, we present results on the number of mutants at a finite population size. In Section \ref{sec:sitefreq}, we introduce the multiple-site model and present results on the site frequency spectrum. In Section \ref{sec:sum}, we discuss the multiple-site model in relation to recent works and data. In Section \ref{sec:proofs}, we present proofs of our main results. See Appendix \ref{sec:app} for a generalisation of the results of Sections \ref{sec:lpsm} and \ref{sec:ltsm}.

\section{Two-type model}
\label{sec:mod}
The wildtype cells grow as a linear birth-death process $(A(t))_{t\geq0}$, with birth and death rates $\alpha_A$ and $\beta_A$ respectively. That is to say $(A(t))_{t\geq0}$ is a continuous time Markov process on $\mathbb{Z}_{\geq0}$ with transition rates
\ba
i\mapsto
\begin{cases}i+1,\quad\text{rate }i\alpha_A;\\i-1,\quad\text{rate }i\beta_A.
\end{cases}
\ea
The initial number of wildtype cells $A(0)\in\mathbb{N}$ is fixed.
The mutation rate is $\nu>0$. Mutation events occur as a Cox process $(K(t))_{t\geq0}$ with intensity $(\nu A(t))_{t\geq0}$. The mutation times are
\[T_i:=\inf\{t\geq0:K(t)=i\}\]
for $i\in\mathbb{N}$.
Each mutation event initiates a clone which grows as a linear birth-death process with birth and death rates $\alpha_B$ and $\beta_B$. Clones grow independently of the wildtype growth and mutation times, and are represented by the i.i.d. processes $(Y_i(t))_{t\geq0}$ for $i\in\mathbb{N}$, with $Y_i(0)=1$. The total mutant population size at time $t$ is
\[B(t)=\sum_{i=1}^{K(t)}Y_i(t-T_i).\]
Write $\lambda_A=\alpha_A-\beta_A$ and $\lambda_B=\alpha_B-\beta_B$ for the fitnesses of the wildtype and mutant cells. We shall only be concerned with the case of supercritical wildtype growth, $\lambda_A>0.$

Note that the process counting the number of cells, $(A(t),B(t))_{t\geq0}$, is a Markov process on $\mathbb{Z}_{\geq0}\times\mathbb{Z}_{\geq0}$, with transition rates
\[(i,j)\mapsto\begin{cases}(i+1,j),\quad\text{rate }i\alpha_A;\\(i-1,j),\quad\text{rate }i\beta_A;\\(i,j+1),\quad\text{rate }i\nu+j\alpha_B;\\(i,j-1),\quad\text{rate }j\beta_B.\\
\end{cases}\]

We are interested in the process at a fixed time $t$, and at the random times
\[\sigma_n:=\inf\{t\geq0:A(t)+B(t)\geq n\}\]
and
\[\tau_n:=\inf\{t\geq0:A(t)\geq n\},\]
for $n\in\mathbb{N}$. Trivially, $\sigma_n\leq\tau_n$.

A classic application of the model is the emergence of drug resistance in cancer. Here, type A and B cells represent drug sensitive and resistant cells respectively. While the age of tumour is typically unknown, its size can be measured. Thus the times $\sigma_n$ are relevant.

Another interpretation of the model is metastasis. Here, type A cells make up the primary tumour, and the clones represent secondary tumours. In this case the times $\tau_n$ are relevant.
\section{Large time and population limits}
\label{sec:lt}
Keeping the mutation rate fixed, the long-term behaviour of the model is mostly already well understood. Durrett and Moseley \cite{dm} study the case $\lambda_A<\lambda_B$. Janson \cite{sj} studies a broad class of urn models, which encompasses Kendall's model in the case $\lambda_A>\lambda_B$. We do not present results as detailed as Janson's. Our aim for this section is not to offer a comprehensive study, but rather bring together basic results which give valuable insight.

First we make note of a classic result:
\bq\lim_{t\rightarrow\infty}e^{-\lambda_At}A(t)=W\label{classic}\eq
almost surely (see \cite{an} or \cite{rdbp}). Here,
\ba
W\overset{d}{=}\sum_{i=1}^{A(0)}\chi_i\psi_i,
\ea
where the $\chi_i\sim$ Bernoulli$(\lambda_A/\alpha_A)$ and the $\psi_i\sim$ Exponential($\lambda_A/\alpha_A$) are independent.
\begin{remark}
The event that the wildtype population eventually becomes extinct agrees with the event $\{W=0\}$ almost surely.
\end{remark}
We see a trichotomy, depending on the relative fitness of wildtype and mutant cells. Part 1 of Theorem \ref{largetime} is a special case of \cite[Theorem 3.1]{sj}, and part 3 is \cite[Theorem 2]{dm}.
\begin{mythe}[Large time limit]\label{largetime}The following limits hold almost surely.
\begin{enumerate}
\item
For $\lambda_A>\lambda_B$,
\[\lim_{t\rightarrow\infty}e^{-\lambda_At}B(t)=\frac{\nu}{\lambda_A-\lambda_B}W.\]
\item
For $\lambda_A=\lambda_B$,
\[\lim_{t\rightarrow\infty}t^{-1}e^{-\lambda_At}B(t)=\nu W.\]
\item
For $\lambda_A<\lambda_B$,
\[\lim_{t\rightarrow\infty}e^{-\lambda_B t}B(t)=V.\]
\end{enumerate}

\end{mythe}
The limit random variable $W$ comes from (\ref{classic}). The limit random variable $V$ is $[0,\infty)$-valued with mean
\[\mathbb{E}[V]=\frac{A(0)\nu}{\lambda_B-\lambda_A}.\]
The full distribution of $V$ is given in \cite[Section 4.3]{ak}, which we do not state here for the sake of brevity.

For $\lambda_A\geq \lambda_B$, conditioned on wildtype non-extinction, any individual clone ultimately makes up zero proportion of the mutant population. That is to say, conditioned on $W>0$,
\[\lim_{t\rightarrow\infty}\frac{Y_i(t-T_i)}{B(t)}=0\]
almost surely. We say that the mutant population is driven by the wildtype growth. This is seen in the limit random variables' dependence on $W$.

For $\lambda_A<\lambda_B$, early arriving clones make an important contribution to the mutant population. Conditioned on $W>0$,
\[\lim_{t\rightarrow\infty}\frac{Y_i(t-T_i)}{B(t)}=\frac{X_ie^{-\lambda_BT_i}}{V}\]
almost surely. Note that if $W>0$, then $V>0$ \cite{dm}. The $X_i$ are i.i.d. with distribution $\chi^B\psi^B$, where $\chi^B\sim$ Bernoulli$(\lambda_B/\alpha_B)$ and $\psi^B\sim$ Exponential($\lambda_B/\alpha_B$) are independent. We say that the mutant population is driven by the clone growth.

To see the asymptotic behaviour of the number of mutations, simply consider $\alpha_B=\beta_B=0$ in Theorem \ref{largetime}:
\[
\lim_{t\rightarrow\infty}e^{-\lambda_At}K(t)=\frac{\nu}{\lambda_A}W
\]
almost surely.

As corollaries to Theorem \ref{largetime} we obtain large population limits. Note that conditioned on $W>0$, $\lim_{n\rightarrow\infty}\tau_n=\lim_{n\rightarrow\infty}\sigma_n=\infty$ almost surely.
\begin{corollary}[Large wildtype population limit]\label{taucor}Conditioned on\\$W>0$, the following limits hold almost surely.
\begin{enumerate}
\item
For $\lambda_A>\lambda_B$,
\[\lim_{n\rightarrow\infty}n^{-1}B(\tau_n)=\frac{\nu}{\lambda_A-\lambda_B}.\]
\item
For $\lambda_A=\lambda_B$,
\[\lim_{n\rightarrow\infty}(n\log(n))^{-1}B(\tau_n)=\frac{\nu}{\lambda_A}.\]
\item
For $\lambda_A<\lambda_B$,
\[\lim_{n\rightarrow\infty}n^{-\lambda_B/\lambda_A} B(\tau_n)=VW^{-\lambda_B/\lambda_A}.\]
\end{enumerate}
\end{corollary}
\begin{corollary}[Large total population limit]\label{sigmacor}Conditioned on $W>0$, the following limits hold almost surely.
\begin{enumerate}
\item
For $\lambda_A>\lambda_B$,
\[\lim_{n\rightarrow\infty}n^{-1}B(\sigma_n)=\frac{\nu}{\lambda_A-\lambda_B+\nu}.\]
\item
For $\lambda_A=\lambda_B$,
\[\lim_{n\rightarrow\infty}n^{-1}\log(n)(n-B(\sigma_n))=\frac{\lambda_A}{\nu}.\]
\item
For $\lambda_A<\lambda_B$,
\[\lim_{n\rightarrow\infty}n^{-\lambda_A/\lambda_B}(n-B(\sigma_n))=V^{-\lambda_A/\lambda_B}W.\]
\end{enumerate}
\end{corollary}
Note that $n-B(\sigma_n)=A(\sigma_n)$.
In case 1, the wildtype and mutant cells come to coexist in a
constant ratio. In cases 2 and 3, the mutant cells eventually
dominate the overall population, with
\bq\lim_{n\rightarrow\infty}n^{-1}B(\sigma_n)=1\label{domi}\eq
almost surely.

\section{Large population small mutation limit}
\label{sec:lpsm}

A tumour may comprise around $10^9$ cells upon detection, with mutation rates per base pair per cell division estimated as $5\times 10^{-10}$ in colorectal cancer \cite{jea}, for example. Hence, a biologically relevant limit can be found by taking the population size to infinity and the mutation rate to zero, while keeping their product fixed.


Suppose that $(\nu_n)_{n\in\mathbb{N}}$ is a sequence of mutation rates satisfying
\bq\lim_{n\rightarrow\infty}n\nu_n=\theta,\label{nucond}\eq
for some $\theta\in(0,\infty)$. For each $n\in\mathbb{N}$, consider the two-type model with mutation rate $\nu_n$. For the wildtype population, mutant population, clone sizes, number of mutations, and mutation times, write $A^{(n)}(\cdot)$, $B^{(n)}(\cdot)$, $Y^{(n)}_i(\cdot)$, $K^{(n)}(\cdot)$, and $T_i^{(n)}$ respectively. Write
\[\sigma'_n:=\inf\{t\geq0:A^{(n)}(t)+B^{(n)}(t)\geq n\}\]
and
\[\tau'_n:=\inf\{t\geq0:A^{(n)}(t)\geq n\}.\]

First we see a connection between the times $\tau_n'$ and $\sigma_n'$ in the large $n$ limit.
\begin{prop}\label{tausigma}Conditioning on $\tau_n'<\infty$,
\ba
\tau_n'-\sigma_n'\rightarrow0
\ea
in probability, as $n\rightarrow\infty$.
\end{prop}
All of our large population small mutation limit results will hold both in terms of the wildtype population and total population size. That is to say, using $\tau_n'$ or $\sigma_n'$ as the time variable will yield the same distributions in the large $n$ limit. To save writing each result twice, we introduce the sequence $(\rho_n)$, which may refer to $(\tau_n')$ or $(\sigma_n')$.

Underlying all subsequent results of this section is that the times of mutation centered about $(\rho_n)$ converge.
\begin{mythe}[Mutations times]\label{times} Conditioning on $\rho_n<\infty$,
\[K^{(n)}(\rho_n+t)\rightarrow K^*(t)\]
in finite dimensional distributions, as $n\rightarrow\infty$. $K^*(t)$ is a Poisson process on $\mathbb{R}$ with intensity $\theta e^{\lambda_At}$.
\end{mythe}

A direct consequence of Theorem \ref{times} is that for each $i\in\mathbb{N}$, conditioning on $\rho_n<\infty$,
\[T_i^{(n)}-\rho_n\rightarrow T_i^*:=\inf\{t\in\mathbb{R}:K^*(t)=i\}\]
in distribution, as $n\rightarrow\infty$. In particular, $T_1^*$ has Gumbel distribution:
\ba\mathbb{P}[T^*_1\geq t]=\exp\left(-\frac{\theta}{\lambda_A} e^{\lambda_At}\right).\label{firsttime}
\ea

Next we look at the number of mutants.
\begin{prop}[Number of mutants]\label{lpsml} Conditioning on $\rho_n<\infty$,
\[
B^{(n)}(\rho_n+t)\rightarrow B^*(t):=\sum_{i=1}^{K^*(t)}Y_i(t-T^*_i)\]
in finite dimensional distributions, as $n\rightarrow\infty$. The $Y_i(\cdot)$ and $K^*(\cdot)$ are independent.
\end{prop}

In particular $B^{(n)}(\rho_n)$ converges in distribution to
\bq
B^*=B^*(0)\overset{d}{=}\sum_{i=1}^{K^*}Y_i(\xi_i),\label{sfsldd}
\eq
where $K^*=K^*(0)\sim\mathrm{Poisson}(\theta/\lambda_A)$, and $\xi_i$ are i.i.d. Exponential($\lambda_A$) random variables independent of the $Y_i(\cdot)$ and $K^*(\cdot)$.

Here $\xi_i$ corresponds to the age of a randomly selected clone, and $Y_i(\xi_i)$ the size of the clone. From \cite[page 109]{an},
\bq\label{bdgf}
\mathbb{E}[z^{Y_i(t)}]=\frac{\beta_B(z-1)-e^{-\lambda_Bt}(\alpha_Bz-\beta_B)}{\alpha_B(z-1)-e^{-\lambda_Bt}(\alpha_Bz-\beta_B)},
\eq
and so
\bq
r(z):&=&\mathbb{E}[z^{Y_i(\xi_i)}]\nonumber\\
&=&\int_0^\infty\mathbb{E}[z^{Y_i(t)}]\lambda_Ae^{-\lambda_At}dt\label{clone}\\
&=&1- (1-q_B)\setlength\arraycolsep{0.5pt}
F\left(\begin{matrix}1,\lambda_A/\lambda_B\\1+\lambda_A/\lambda_B\end{matrix};\frac{q_B-z}{1-z}\right).\nonumber
\eq
The function $F$ is Gauss's hypergeometric function , and $q_B=\beta_B/\alpha_B$, which is a clone's ultimate extinction probability if $q_B\leq1$. The third equality of (\ref{clone}) can be seen by making a change of variable $s=e^{-\lambda_B t}$, and then using a standard integral representation for $F$ (for example \cite[C.8]{ka}).

Clearly $B^*$ is a compound Poisson random variable (\ref{sfsldd}), and has generating function
\bq
\mathbb{E}[z^{B^*}]=\exp\left(\frac{\theta}{\lambda_A} (r(z)-1)\right). \label{mutantdist}
\eq
This recovers recent results of Kessler and Levine \cite{kl} who provided a heuristic derivation of this expression, and Keller and Antal \cite{ka} who derived it for a deterministic exponentially growing wildtype population. Its large $\theta$ limit appeared in Durrett and Moseley \cite{dm} for $\lambda_A<\lambda_B$ (see \cite{ka} for a discussion). If $\alpha_B=\lambda_A$ and $\beta_B=0$, (\ref{mutantdist}) reduces to the Luria-Delbr{\"u}ck distribution \cite{lc}:
\[
\mathbb{E}[z^{B^*}]=(1-z)^{\frac{\theta}{\lambda_A}(z^{-1}-1)}.
\]

\begin{remark}\label{pl}
For $\lambda_B>0$, the generating functions (\ref{clone}) and (\ref{mutantdist}) yield power law tails:
\[\lim_{k\rightarrow\infty}k^{1+\lambda_A/\lambda_B}\mathbb{P}[Y_i(\xi_i)=k]=\frac{\lambda_A}{\lambda_B}(1-q_B)^{1-\lambda_A/\lambda_B}\Gamma(1+\lambda_A/\lambda_B)\]
and
\[\lim_{k\rightarrow\infty}k^{1+\lambda_A/\lambda_B}\mathbb{P}[B^*=k]=\frac{\theta}{\lambda_B}(1-q_B)^{1-\lambda_A/\lambda_B}\Gamma(1+\lambda_A/\lambda_B),\]
which are given in \cite{na,ka,kl}.
\end{remark}

Of potential interest is the number of clones of a given size, perhaps above some lower limit for reliable detection. Let $I$ be a subset of $\mathbb{Z}_{\geq0}$. Consider
\ba
C_I^{(n)}(t)=\sum_{i=1}^{K^{(n)}(t)}1_{\left\{Y^{(n)}_i(t-T_i^{(n)})\in I\right\}}(t),
\ea
giving the number of clones whose size is in $I$ at time $t$.
\begin{prop}[Number of clones of a given size]\label{cs}
Conditioning on $\rho_n<\infty$,
\ba
C^{(n)}_I(\rho_n)\rightarrow C_I^*\sim\mathrm{Poisson}\left(\frac{\theta}{\lambda_A}\mathbb{P}[Y_i(\xi_i)\in I]\right)
\ea
in distribution, as $n\rightarrow\infty$.
\end{prop}

Consider
\ba
M^{(n)}(t)=\max_{1\leq i\leq K^{(n)}(t)}Y^{(n)}_i(t-T_i^n),
\ea
giving the size of the largest clone at time $t$. 
\begin{prop}[Size of largest clone]\label{lc}
Conditioning on $\rho_n<\infty$,
\ba
M^{(n)}(\rho_n)\rightarrow M^*=\max_{1\leq i\leq K^*}Y_i(\xi_i)
\ea
in distribution, as $n\rightarrow\infty$. Here, $\mathbb{P}[M^*\leq k]=\exp\left(-\frac{\theta}{\lambda_A}\mathbb{P}[Y_i(\xi_i)>k]\right)$.
\end{prop}
For an example we take the simplest choice of mutant cell growth: $\beta_B=0$ and $\alpha_B=\lambda_A$. The number of clones above size $k$ is 
\ba
C^*_{\{i\in\mathbb{N}:i\geq k\}}\sim\mathrm{Poisson}\left(\frac{\theta}{\lambda_Ak}\right).
\ea
The size of the largest clone is
\ba
\mathbb{P}[M^*\leq k]=\exp\left(-\frac{\theta}{\lambda_A(k+1)}\right).
\ea

\begin{remark}
In this section we have considered a limit in which the product of the population size and mutation rate, $\theta=n\nu$, remains finite. It should be noted that alternative limits are also possible here. For example, Kessler and Levine \cite{kl1} investigate large $\theta$. In a different twist, Hamon and Ycart \cite[Theorem 1.1]{hy} take the initial population size to infinity, the time of measurement to infinity, and the mutation rate to zero.
\end{remark}

\section{Large time small mutation limit}
\label{sec:ltsm}Here we investigate results similar to Section \ref{sec:lpsm}, but with a view to approximating the process at a fixed time rather than population size. Let $(t_n)$ be a sequence of non-random times converging to infinity, and $(\nu_n)$ a sequence of mutation rates satisfying
\[\lim_{n\rightarrow\infty}e^{\lambda_At_n}\nu_n=\eta,\]
for some $\eta\in(0,\infty)$. For each $n\in\mathbb{N}$ consider the two-type model with mutation rate $\nu_n$. We use the superscript $(n)$ notation established in Section \ref{sec:lpsm}.
\begin{prop}[Mutation times]\label{ltsmt}
As $n\rightarrow\infty$,
\[K^{(n)}(t_n+t)\rightarrow K^\circ(t)\]
in finite dimensional distributions. $K^\circ(t)$ is a Cox process on $\mathbb{R}$ with intensity $W\eta e^{\lambda_At}$, where $W$ is distributed as (\ref{classic}).
\end{prop}

A direct consequence of Proposition \ref{ltsmt} is that for each $i\in\mathbb{N}$,
\[T_i^{(n)}-t_n\rightarrow T^\circ_i:=\inf\{t\in\mathbb{R}:K^\circ(t)=i\}\]
in distribution, as $n\rightarrow\infty$.
\begin{prop}[Number of mutants]\label{ltsml}
As $n\rightarrow\infty$,
\[B^{(n)}(t_n+t)\rightarrow B^\circ(t)=\sum_{i=1}^{K^\circ(t)}Y_i(t-T^\circ_i),\]
in finite dimensional distributions. The $Y_i(\cdot)$ are independent of $K^\circ(\cdot)$.
\end{prop}
Observe that
\bq B^\circ=B^\circ(0)\overset{d}{=}\sum_{i=1}^{K^\circ}Y_i(\xi_i),\label{Bcirc}\eq
where $K^\circ=K^\circ(0)$ conditioned on $W$ is Poisson distributed with mean $W\eta/\lambda_A$. The generating function of $B^\circ$ is
\bq\label{mdltsm}
\mathbb{E}[z^{B^\circ}]&=&\mathbb{E}\left[\exp\left(\frac{W\eta}{\lambda_A}(r(z)-1)\right)\right]\\
&=&\left(\frac{\lambda_A^2-\beta_A\eta(r(z)-1)}{\lambda_A^2-\alpha_A\eta(r(z)-1)}\right)^{A(0)},\nonumber
\eq
where $r(z)$ is the clone size generating function, given by (\ref{clone}).

\begin{remark}\label{plt}For $\lambda_B>0$, the generating function (\ref{mdltsm}) yields the same power-law tail as (\ref{clone}) and (\ref{mutantdist}) (see Remark \ref{pl}):
\[
\lim_{k\rightarrow\infty}k^{1+\lambda_A/\lambda_B}\mathbb{P}[B^\circ=k]=\frac{A(0)\eta}{\lambda_B}(1-q_B)^{1-\lambda_A/\lambda_B}\Gamma(1+\lambda_A/\lambda_B).
\]
 
\end{remark}
The number of clones of a given size and the size of the largest clone can be seen in the large time small mutation limit. Simply replace $K^*$ with $K^\circ$ in Propositions \ref{cs} and \ref{lc}.

Finally, we comment that the large time small mutation limit justifies a common approximation of the model, in which the wildtype population grows as $(We^{\lambda_At})_{t\in\mathbb{R}}$. Here $B^\circ(\cdot)$ corresponds to $Z^*_1(\cdot)$ defined in \cite{dm}, for example.

\section{Finite size results}
\label{sec:finite}

For simplicity, we consider $(A(0),B(0))=(1,0)$ in this section. However, it should not be too difficult to extend to arbitrary initial cell numbers.

We are able to give the distribution of $B(\tau_n)$ in the special case of no wildtype cell death.
\begin{prop}\label{finite}For $\beta_A=0$,
\bq
B(\tau_n)\overset{d}{=}\sum_{i=1}^{n-1}\sum_{j=1}^{K_i(\xi_i)}Y_{i,j}(U_{i,j}\xi_i),\label{yule}
\eq
where $(K_i(t))_{t\geq0}$ are Poisson processes with intensity $\nu$, $Y_{i,j}(\cdot)\overset{d}{=}Y_i(\cdot)$, $\xi_i\sim$Exponential($\alpha_A$), and $U_{i,j}\sim$Uniform[0,1], which are all independent. 
\end{prop}
To interpret (\ref{yule}), let's consider a randomly selected type $A$ cell, labelled $i$, of the $n-1$ cells present just before time $\tau_n$. The cell has been alive for time $\xi_i$, and initiated $K_i(\xi_i)$ mutant clones, with mutation times $(1-U_{i,j})\xi_i$ for $j=1,2,..,K_i(\xi_i)$. The clone sizes are $Y_{i,j}(U_{i,j}\xi_i)$.

The mean number of mutant cells at time $\tau_n$ is
\ba
\mathbb{E}[B(\tau_n)]=
  \begin{cases}\frac{(n-1)\nu}{\alpha_A-\lambda_B}, & \lambda_B<\alpha_A;\\
  \infty, & \lambda_B\geq\alpha_A.
\end{cases}
\ea
The generating function of $B(\tau_n)$ is 
\ba\label{btauexact}
\mathbb{E}[z^{B(\tau_n)}]&=&\left[\int_0^\infty\alpha_Ae^{-\alpha_At}\exp\left(\nu t\int_0^1\mathbb{E}[z^{Y_{i,j}(ut)}]-1du\right)dt\right]^{n-1}\\
&=&\left[\frac{1}{1+\frac{\lambda_B\nu}{\alpha_A\alpha_B}}\setlength\arraycolsep{0.5pt}
F\left(\begin{matrix}1,\nu/\alpha_B\\1+\nu/\alpha_B+\alpha_A/\lambda_B\end{matrix};\frac{q_B-z}{q_B-1}\right)\right]^{n-1},
\ea
where $\mathbb{E}[z^{Y_{i,j}(ut)}]$ is given by (\ref{bdgf}). The computation is lengthy but straightforward; one can apply the integral expression \cite[C.8]{ka} for the hypergeometric function, and the identity \cite[C.10]{ka}.
As in Remarks \ref{pl} and \ref{plt}, for $\lambda_B>0$,
\bq\label{infinitemoments}
\lim_{k\rightarrow\infty}k^{1+\alpha_A/\lambda_B}\mathbb{P}[B(\tau_n)=k]\in(0,\infty)
\eq
exists. The limit can be obtained using the method of \cite[Section 6]{ka} (which is based on \cite{fs}), but is too cumbersome to include here. Power-law tails have often appeared in two-type branching processes, but were generally considered to be an artefact of approximation \cite{dm,qz}. 

\begin{remark}Contrary to (\ref{infinitemoments}), moments of $B(\tau_n)$ are finite in the standard semi-deterministic version of the model (e.g. \cite{lc} and \cite{ka}).
\end{remark}

Next, specialising further to neglect wildtype and mutant death, we connect the distributions of the $B(\sigma_n)$ and $B(\tau_n)$.
\begin{lemma}\label{connect}
For $\beta_A=\beta_B=0$, and integers $0\leq k<n$,
\ba
\mathbb{P}[B(\sigma_n)\leq k]=\mathbb{P}[B(\tau_{n-k})\leq k].
\ea
\end{lemma}
A similar result was given by Janson \cite[Lemma 9.1]{sjt} for a different class of urn models. Although Lemma \ref{connect} can be combined with Proposition \ref{finite} to determine the distribution of $B(\sigma_n)$, it does not seem likely that a tractable explicit expression can be obtained in general. However, for neutral mutations, Angerer was able to solve a recursion for the probabilities $\mathbb{P}[B(\sigma_n)=k]$ \cite[Corollary 2.2]{wa}.
\begin{prop}[Angerer]\label{wafs}For $\alpha_A+\nu=\alpha_B$ and $\beta_A=\beta_B=0$,
\[\mathbb{P}[B(\sigma_n)=k]=\sum_{i=1}^{n-k}(-1)^{n-i}\binom{n-k-1}{i-1}\binom{i\frac{\alpha_A}{\alpha_B}-1}{n-1}.\]
\end{prop}

\section{Multiple site model and site frequency spectrum}
\label{sec:sitefreq}

In the case of neutral mutations, we extend the two-type model to consider mutations at multiple sites on the genome.

The overall population $(C(t))_{t\geq0}$ grows as a birth-death branching process. Cells divide and die at rates $a$ and $b$, where $a>b$. Consider $S$ sites, labelled $i\in\{1,..,S\}$. Each cell is labelled by some $(z_1,..,z_S)\in\{0,1\}^S$, where $z_i=1$ corresponds to a mutation at site $i$. Initially there are an arbitrary number of cells all with label $(0,..,0)$.

The mutations are modelled to occur in such a way that the number of mutants with respect to a particular site follows the two-type model. At each division event the parent cell dies, and two daughter cells are produced. The daughter cells inherit the parent's mutations and may receive further mutations. Suppose that site $i$ is not already mutated in the parent cell. With probability $1-\mu$ site $i$ does not receive a mutation in either daughter cell. With probability $\mu$ exactly one of the daughter cells receives a mutation at site $i$.

To state this more precisely, let us consider a parent cell with label $(z_1,..,z_S)\in\{0,1\}^S$ dividing. The two daughter cells have labels $(Z^{[1]}_1,..,Z^{[1]}_S)$ and $(Z^{[2]}_1,..,Z^{[2]}_S)$, where for each $i$
\ba
(Z^{[1]}_i,Z^{[2]}_i)=\begin{cases}(z_i,z_i),\quad&\text{probability }1-\mu;\\
(\min\{z_i+1,1\},z_i),\quad&\text{probability }\mu/2;\\
(z_i,\min\{z_i+1,1\}),\quad&\text{probability }\mu/2.
\end{cases}
\ea
\begin{remark}For our purposes, we do not need to specify the joint distribution of $(Z^{[1]}_i,Z^{[2]}_i)_{i=1}^S$.
\end{remark}
\begin{remark}As in Kendall's model, we neglect back mutations, and neglect the event that a cell division sees both daughter cells receiving the same mutation (see \cite{kwb} for a biological justification).
\end{remark}
For each $i\in\{1,..,S\}$, let $B_i(t)$ be the number of cells at time $t$ with $z_i=1$.

Now we establish the connection between the multiple-site model and two-type model.
Put $\alpha_A+\nu=\alpha_B=a$, $\beta_A=\beta_B=b$, $\nu=\mu a$. Then for each $i$,
\[(C(t)-B_i(t),B_i(t))_{t\geq0}\overset{d}{=}(A(t),B(t))_{t\geq0}.\]

The site frequency spectrum is defined to be the number of sites who see mutations in a given number of cells, i.e. the sequence 
\[\left(\sum_{i=1}^S1_{\{B_i(t)=k\}}\right)_{k\in\mathbb{Z}_{\geq0}}.\]
By linearity of expectation, the mean site frequency spectrum is determined by
\bq
\mathbb{E}\left[\sum_{i=1}^S1_{\{B_i(t)=k\}}\right]=S\mathbb{P}[B(t)=k].\label{SFS}
\eq
Antal and Krapivsky \cite{ak} found the distribution of $B(t)$, by solving the Kolmogorov equations. For brevity, we do not state their result. To see the mean site frequency spectrum at a fixed population size, define
\[\sigma_n:=\inf\{t\geq0:C(t)\geq n\},\]
as in the two-type model.
Then for $b=0$ and $C(0)=1$,
\[
\mathbb{E}\left[\sum_{i=1}^S1_{\{B_i(\sigma_n)=k\}}\right]=S\sum_{i=1}^{n-k}(-1)^{n-i}\binom{n-k-1}{i-1}\binom{i\frac{\alpha_A}{\alpha_B}-1}{n-1},
\]
by Proposition \ref{wafs}.

Let's return to the general setting of $b\geq0$ and $C(0)\in\mathbb{N}$. We briefly comment on the long term behaviour of the site frequency spectrum. The number of sites who are mutated in a given number of cells converges to zero: for any $k\geq1$
\[\lim_{t\rightarrow\infty}\sum_{i=1}^S1_{\{B_i(t)=k\}}=0,
\]
almost surely. Let $x\in[0,1)$. The number of sites who are mutated in at least proportion $x$ of the population converges to $S$:
\[\lim_{n\rightarrow\infty}\sum_{i=1}^S1_{\{B_i(\sigma_n)\geq xn\}}=S,
\]
almost surely, due to (\ref{domi}).

Next we look at the large population/time small mutation limits. Take a sequence of mutation probabilities $(\mu_n)_{n\in\mathbb{N}}$. For each $n\in\mathbb{N}$ consider the multiple-site model with mutation probability $\mu_n$, birth rate $a$, and death rate $b_n=b(1-\mu_n)$ (with $\lambda=a-b>0$). Write $C^{(n)}(t)$ for the population size and $B_i^{(n)}(t)$ for number of site $i$ mutants at time $t$. Write
\[\sigma_n':=\inf\{t\geq0:C^{(n)}(t)\geq n\}.\]
The purpose of choosing the sequence of death rates $(b_n)$ in this way is to allow for a straightforward adaptation of the two-type results.

\begin{prop}[Large population small mutation limit]\label{sfsp}
Suppose that $(\mu_n)$ satisfies
\[\lim_{n\rightarrow\infty}n\mu_na=\theta,
\]
for some $\theta\in(0,\infty)$. Then
\[
\lim_{n\rightarrow\infty}\mathbb{E}\left[\sum_{i=1}^S1_{\{B^{(n)}_i(\sigma'_n)=k\}}\big|\sigma'_n<\infty\right]=S\mathbb{P}[B^*=k],
\]
where $B^*$ is distributed according to (\ref{sfsldd}) with $\alpha_A=\alpha_B=a$ and $\beta_A=\beta_B=a$.
\end{prop}

\begin{prop}[Large time small mutation limit]\label{sfst}
Take a sequence of times $(t_n)$ converging to infinity, with
\[\lim_{n\rightarrow\infty}e^{\lambda t_n}\mu_na=\eta,\]
for some $\eta\in(0,\infty)$. Then
\[
\lim_{n\rightarrow\infty}\mathbb{E}\left[\sum_{i=1}^S1_{\{B^{(n)}_i(t_n)=k\}}\right]=S\mathbb{P}[B^\circ=k],
\]
where $B^\circ$ is distributed according to (\ref{Bcirc}) with $\alpha_A=\alpha_B=a$ and $\beta_A=\beta_B=a$.
\end{prop}
\begin{remark}\label{pltsfs}Our approximations for the mean site frequency spectrum have power-law tails:
\[
\lim_{k\rightarrow\infty}k^2S\mathbb{P}[B^*=k]=\frac{S\theta}{\lambda}
\]
and
\[
\lim_{k\rightarrow\infty}k^2S\mathbb{P}[B^\circ=k]=\frac{S\eta C(0)}{\lambda},
\]
which are special cases of Remarks \ref{pl} and \ref{plt}.
\end{remark}

Since the size, rather than age, of a tumour can be observed, we are most interested in the large population small mutation limit. To give the reader an idea of its appearance, in Figure \ref{fig} the mean site frequency spectrum as given by Proposition \ref{sfsp} is plotted. The theoretical result is compared to simulations, with birth, death and scaled mutation rates taken from biological literature. In particular, we consider $a=0.25$ and $b=0.18$ (per day), which were estimated in colorectal cancers by \cite{djla}. According to \cite{jea}, $\theta$ may be of the order of $a$; we consider two different values for $\theta$ in this region. We take a relatively small population size of $n=10^3$ and number of sites $S=50$, so that computation time is reasonable. It is expected that taking larger $n$ and fixed $\theta$ will give an even closer fit between theory and simulations.
\begin{figure}
\caption{Simulated and theoretical expected site frequency spectrum, with $a=0.25$, $b=0.18$, $S=50$, $C(0)=1$, $n=10^3$. Two different mutation rates are plotted: $\mu=10^{-3}$ (left) and $\mu=10^{-2}$ (right). The average has been taken over $10^4$ simulations in each case. }\label{fig}
\includegraphics[width=.49\textwidth]{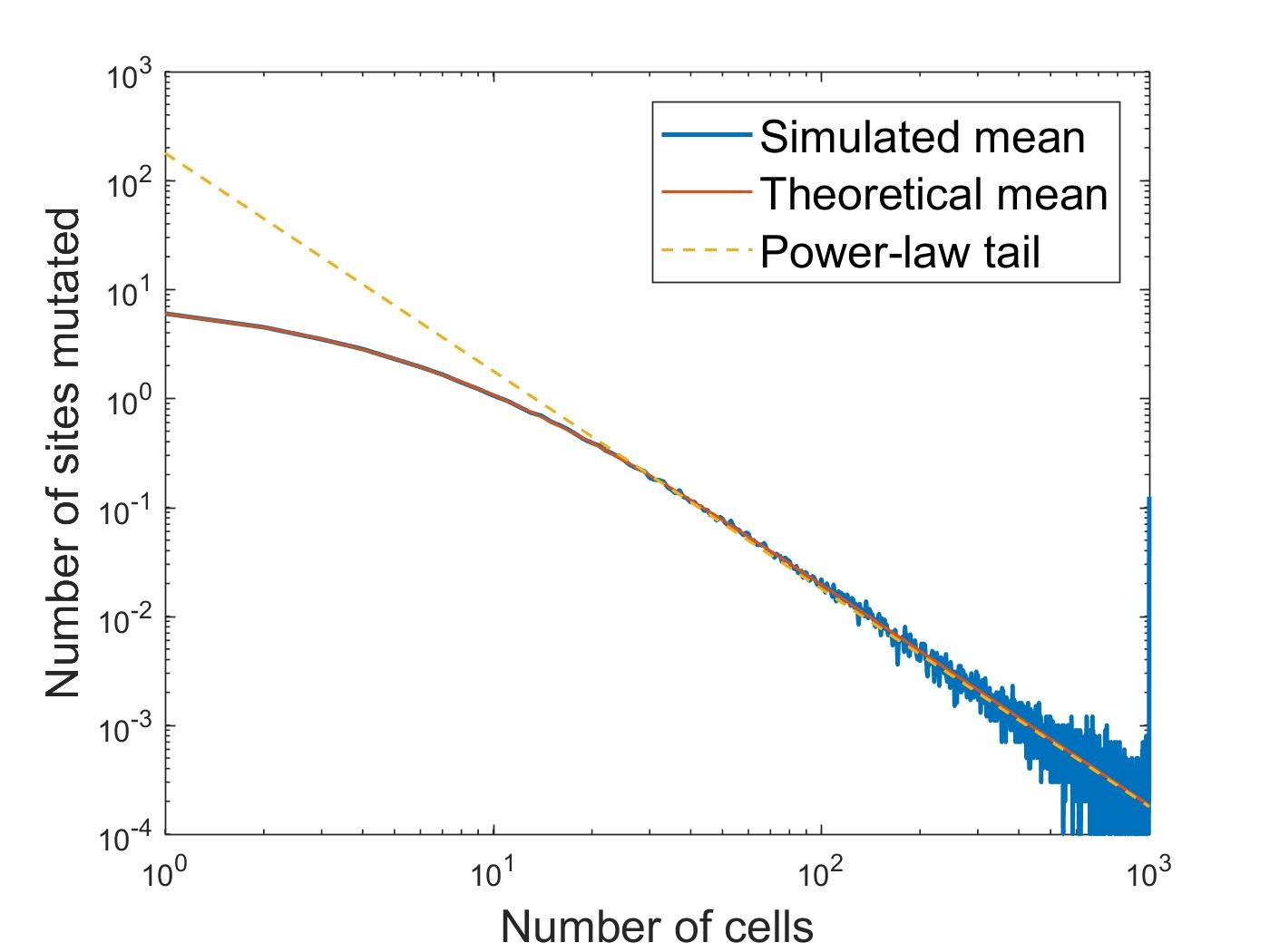}
\includegraphics[width=.49\textwidth]{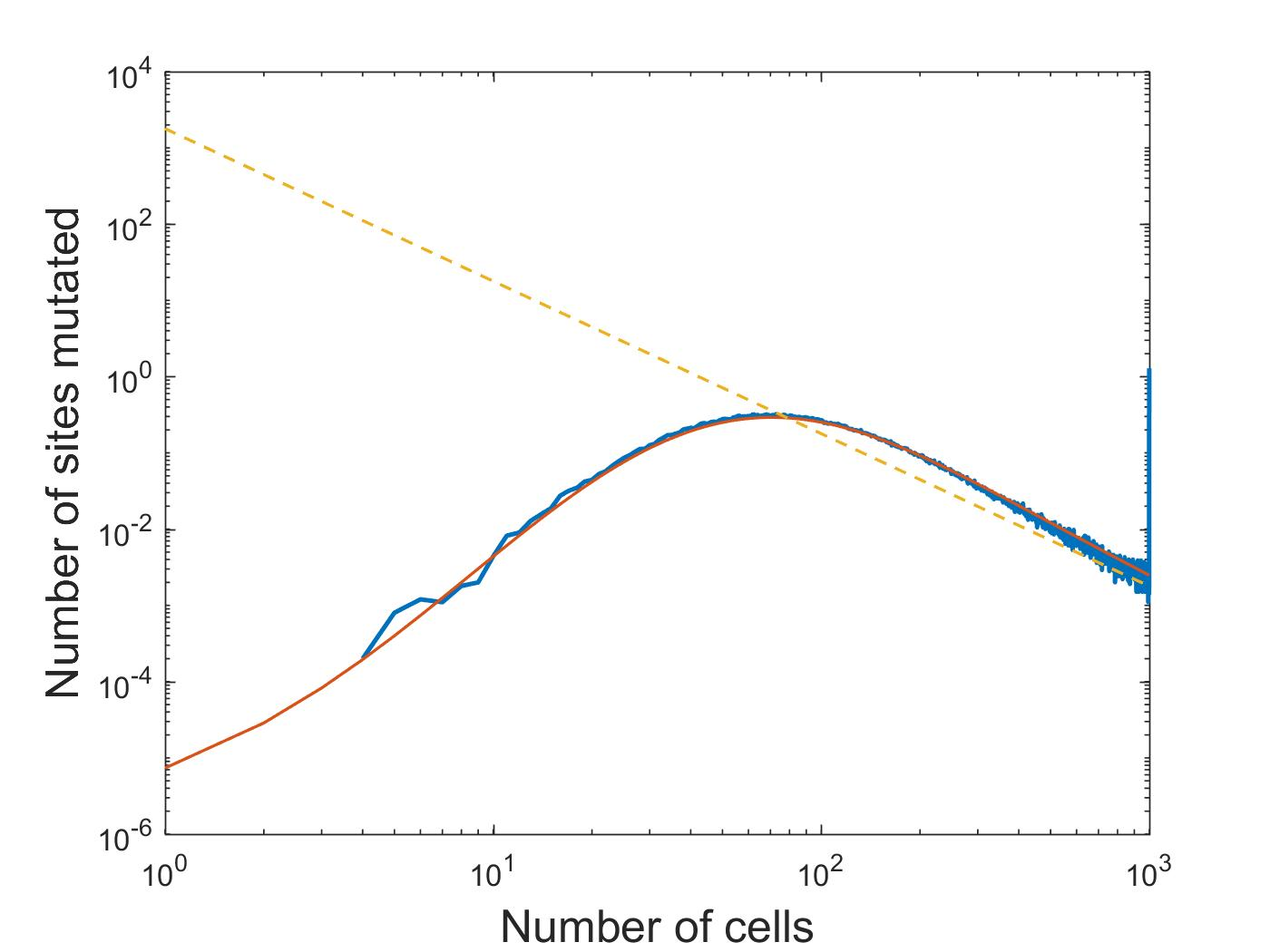}
\end{figure}

\section{Discussion}

\label{sec:sum}

From a single cancer cell a tumour may grow to comprise billions of cells. Mutations can occur at cell divisions, ultimately leading to great genetic diversity within a single tumour. With the advent of next-generation DNA sequencing, vast quantities of cancer genomes have been sequenced. Data has been made publicly available through the Cancer Genome Atlas and International Cancer Genome Consortium, for example. Considerable efforts have been made in recent years to explain observed mutation patterns with mathematical models, and from the observed mutation patterns to infer the evolutionary history of tumours.

Striking examples are Williams et al. \cite{gs} and Bozic et al. \cite{ib}, who consider deterministic and branching process models respectively. They both derive that the expected frequency of mutations occurring in $x$ proportion of cells has density proportional to $x^{-2}$ (away from $0$). In \cite{gs}, 323 out of 904 cancers considered are deemed to fit the $x^{-2}$ power-law. In \cite{ib}, 14 out of 42 cancers are deemed to fit the power-law.

The models of \cite{gs,ib,oi,durr} all used the infinite sites assumption, which states that each site can mutate at most once over the lifetime of a tumour. Statistical analysis of cancer genomic data refutes this assumption \cite{isr}. Furthermore, we make a theoretical argument against the infinite sites assumption in the branching process setting. According to Proposition \ref{times}, the number of times a particular site has mutated before the population size reaches $n$ is approximately Poisson($n\nu/\lambda_A$). Therefore the infinite sites simplification may be appropriate when $n\nu/\lambda_A$ is much smaller than $1$. However \cite{gs} estimated effective mutation rates, $\nu/\lambda_A$, of single base pairs to be in the region of $10^{-7}-10^{-6}$. If a detected tumour comprises $10^8-10^9$ cells (e.g. \cite{ib}), then $n\nu/\lambda_A$ is not sufficiently small.

In Proposition \ref{sfsp}, we have shown that the mean site frequency spectrum can be approximated by a well known generalisation of the Luria-Delbr{\"u}ck distribution. The distribution's $x^{-2}$ tail agrees with theoretical predictions and data in \cite{gs,ib}. But our predictions disagree at the lower end of the frequency spectrum. Due to unreliable data, \cite{gs,ib} did not make a model-data comparison for mutations occurring in less than $10\%$ of cells.

In upcoming work we extend the multiple-site model to non-neutral mutations.

\section{Proofs}
\label{sec:proofs}

\subsection*{Proofs for Section \ref{sec:lt}}
\begin{proof}[Proof of Theorem \ref{largetime}]
For part 2, one needs to observe that
\[
\left(e^{-\lambda_A t}B(t)-te^{-\lambda_A t}\nu A(t)\right)_{t\geq0}
\]
is a martingale with respect to the obvious filtration, and is bounded in $L_2$.

For part 1, the reader may refer to \cite{sj} for a full proof in a more general and notation-heavy setting. For the reader's convenience, we offer the essence of Janson's proof here. Crucially,
\[
\left(M(t)\right)_{t\geq0}=\left(e^{-\lambda_B t}B(t)-\frac{\nu}{\lambda_A-\lambda_B}e^{-\lambda_B t}A(t)\right)_{t\geq0}
\]
is a martingale. Janson obtains bounds for the probabilities
\[
\mathbb{P}\left[\sup_{t\in[n-1,n]}\left|e^{(\lambda_B-\lambda_A)t}M(t)\right|>\epsilon\right],
\]
via Doob's martingale inequality, and then applies the Borel-Cantelli lemma.
\end{proof}
\begin{proof}[Proof of Corollary \ref{sigmacor}, part 2]
First, rewrite
\ba
\frac{\log(n)A(\sigma_n)}{n}&=&\frac{\log\left(A(\sigma_n)+B(\sigma_n)\right)A(\sigma_n)}{A(\sigma_n)+B(\sigma_n)}\\
&=&\frac{1}{\sigma_n}\left[\log\left(\frac{A(\sigma_n)+B(\sigma_n)}{\sigma_ne^{\lambda_A\sigma_n}}\right)+\log(\sigma_n)+\lambda_A\sigma_n\right]\\
&&\times\frac{e^{-\lambda_A\sigma_n}A(\sigma_n)}{\sigma_n^{-1}e^{-\lambda_A\sigma_n}\left(A(\sigma_n)+B(\sigma_n)\right)}.
\ea
Then apply Theorem \ref{largetime} and (\ref{classic}), to take $n\rightarrow\infty$.
\end{proof}
The remaining parts of Corollaries \ref{taucor} and \ref{sigmacor} can be proven in a similar manner.

\subsection*{Proofs for Sections \ref{sec:lpsm} and \ref{sec:ltsm}}~\\

For each $n\in\mathbb{N}$, the joint distribution of
\bq
(A^{(n)}(\cdot), B^{(n)}(\cdot), (Y^{(n)}_i(\cdot))_{i\in\mathbb{N}}, K^{(n)}(\cdot),(T_i^{(n)})_{i\in\mathbb{N}}, \sigma_n',\tau_n')\label{srv}
\eq
has been specified, with respect to the mutation rate $\nu_n$. Note that the distributions of $A^{(n)}(\cdot)$ and $Y_i^{(n)}(\cdot)$ do not depend on $n$. We will construct the sequence (\ref{srv}) ranging over $n\in\mathbb{N}$ on a single probability space $(\Omega,\mathcal{F},\mathbb{P})$ in a way that allows weak convergence to be shown via almost sure convergence.

On $(\Omega,\mathcal{F},\mathbb{P})$ define the independent processes $(A(t))_{t\geq0}$, $(Y_i(t))_{t\geq0}$ for $i\in\mathbb{N}$, and $(N(t))_{t\geq0}$. As one would expect we take $A(\cdot)\overset{d}{=}A^{(n)}(\cdot)$ and $Y_i(\cdot)\overset{d}{=}Y_i^{(n)}(\cdot)$. Take $N(\cdot)$ to be a Poisson counting process with intensity $1$.

Define the mutation counting process by
\[K^{(n)}(t)=N\left(\int_0^t\nu_nA(s)ds\right).\]
The mutation times are given by
\[T_i^{(n)}=\inf\{t\geq0:K^{(n)}(t)=i\}.\]
The total mutant population is
\[B^{(n)}(t)=\sum_{i=1}^{K^{(n)}(t)}Y_i(t-T_i^{(n)}).\]
So the only dependence on $n$ comes from the mutation times.
As before, define
\[\sigma'_n=\inf\{t\geq0:A(t)+B^{(n)}(t)\geq n\}\]
and
\[\tau'_n=\inf\{t\geq0:A(t)\geq n\}.\]

The large population small mutation limit results all involve conditioning on $\sigma_n'<\infty$ or $\tau_n'<\infty$. Lemmas \ref{conditioning} and \ref{we} will demonstrate that the results can be equivalently formulated by instead conditioning on non-extinction of the wildtype population.

\begin{mylem}\label{conditioning}
Suppose that $(E_n)_{n\in\mathbb{N}}$ and $(F_n)_{n\in\mathbb{N}}$ are sequences of events, such that
\begin{enumerate}
\item
$\forall n\in\mathbb{N}(F_n\supset F_{n+1})$,
\item
$\cap_{n\in\mathbb{N}}F_n=F$, and
\item
$\mathbb{P}[F]>0$.
\end{enumerate}

Then
\[
\lim_{n\rightarrow\infty}\mathbb{P}[E_n|F_n]=\lim_{n\rightarrow\infty}\mathbb{P}[E_n|F],
\]
if it exists.
\begin{proof}
Write
\[
\mathbb{P}[E_n|F_n]=\frac{\mathbb{P}[F]}{\mathbb{P}[F_n]}\mathbb{P}[E_n|F]+\frac{\mathbb{P}[E_n\cap F_n\backslash F]}{\mathbb{P}[F_n]},
\]
and take $n\rightarrow\infty$.
\end{proof}
\end{mylem}

\begin{lemma}\label{we}
\[\{W>0\}=\cap_{n\in\mathbb{N}}\{\tau_n'<\infty\}=\cap_{n\in\mathbb{N}}\{\sigma_n'<\infty\},\]
where $W$ is given by (\ref{classic}).
\begin{proof}
That $\{W>0\}\subset\cap_{n\in\mathbb{N}}\{\tau_n'<\infty\}$ and $\cap_{n\in\mathbb{N}}\{\tau_n'<\infty\}\subset\cap_{n\in\mathbb{N}}\{\sigma_n'<\infty\}$ should be clear. We show that
\[\cap_{n\in\mathbb{N}}\{\sigma_n'<\infty\}\subset\{W>0\}.\]
Indeed, fix $\omega\in\{W=0\}=\{\exists t\geq0,A(t)=0\}$.

Then
\[\int_0^\infty A(t)dt<\infty.\]
So one can choose sufficiently large $n\in\mathbb{N}$ such that both
\[\nu_n\int_0^\infty A(t)dt<\sup\{t\geq0:N(t)=0\},\]
and
\[\tau_n'=\infty.\]
In this case we must have
\[\omega\in\{\forall t\geq0,K^{(n)}(t)=0\}\cap\{\tau_n'=\infty\}\subset\{\sigma_n'=\infty\}.\]
\end{proof}
\end{lemma}
We will now prove the results of Section \ref{sec:lpsm} by conditioning on $W>0$.

\begin{mylem}\label{easytimes}
Conditioning on $W>0$, as $n\rightarrow\infty$,
\[K^{(n)}(\tau_n'+t)=N\left(\int_{-\tau_n'}^t\nu_nA(\tau_n'+s)ds\right)\rightarrow N\left(\int_{-\infty}^t\theta e^{\lambda_As}ds\right)=:K^*(t)\]
and
\[T_i^{(n)}-\tau_n'\rightarrow T_i^*=\inf\{t\in\mathbb{R}:K^*(t)=i\}\]
almost surely, for each $t\in\mathbb{R}$.
\begin{proof}

That $A(\cdot)$ is cadlag and satisfies (\ref{classic}), are enough to see that
\[
\sup_{t\geq0}\frac{A(t)}{e^{\lambda_At}}<\infty,
\]
and
\[
\sup_{n\in\mathbb{N}}\frac{e^{\lambda_A\tau_n'}}{A(\tau_n')}<\infty
\]
almost surely. Now write
\ba
\nu_nA(\tau_n'+t)=n\nu_n\frac{A(\tau_n'+t)}{e^{\lambda_A(\tau_n'+t)}}\frac{e^{\lambda_A\tau_n'}}{A(\tau_n')}e^{\lambda_At}.
\ea
It becomes apparent that
\[\lim_{n\rightarrow\infty}\nu_nA(\tau_n'+t)= \theta e^{\lambda_At},\]
and for all $t\in\mathbb{R}$
\[\sup_{n\in\mathbb{N}}\nu_nA(\tau_n'+t)\leq Le^{\lambda_At}\label{t}\]
almost surely, for some positive random variable $L$. Then, using dominated convergence and the fact that $N(\cdot)$ is almost surely continuous at $\int_{-\infty}^t\theta e^{\lambda_As}ds$, we are done.
\end{proof}
\end{mylem}
Lemma \ref{easytimes} corresponds to Theorem \ref{times} in the case $(\rho_n)=(\tau_n')$. Lemmas \ref{suplemma} and \ref{tslemma} extend the result to $(\rho_n)=(\sigma_n')$.
\begin{mylem}\label{suplemma}Conditioning on $W>0$,
\[\sup_{n\in\mathbb{N}}B^{(n)}(\sigma_n')<\infty\]
almost surely.
\begin{proof}
For each $n\in\mathbb{N}$ consider the process
\[\hat{B}^{(n)}(t)=\sum_{i=1}^{K^{(n)}(t)}\hat{Y}_i(t-T_i^{(n)}),\]
where
\[\hat{Y}_i(t)=\sup_{s\in[0,t]}Y_i(s).\]
With probability $1$,
\[\lim_{n\rightarrow\infty}\hat{B}^{(n)}(\tau_n')\in[0,\infty)\]
exists, by Lemma \ref{easytimes} and the almost sure continuity of the $\hat{Y}_i$ at $-T^*_i$. Finally,
\[B^{(n)}(\sigma_n')\leq\hat{B}^{(n)}(\sigma_n')\leq\hat{B}^{(n)}(\tau_n')\leq\sup_{n\in\mathbb{N}}\hat{B}^{(n)}(\tau_n')<\infty,\]
using the monotonicity of $\hat{B}^{(n)}(\cdot)$.
\end{proof}
\end{mylem}
\begin{mylem}\label{tslemma}
Conditioning on $W>0$,
\[\tau_n'-\sigma_n'\rightarrow0\]
almost surely.
\end{mylem}
\begin{proof}
Consider a positive sequence $(a_n)$, such that
\begin{enumerate}
\item
$\lim_{n\rightarrow\infty}a_n=\infty$, and
\item
$\lim_{n\rightarrow\infty}(n-a_n)/n=1$.
\end{enumerate} For example $a_n= n^{1/2}$ will do.
Since
\ba
e^{\lambda_A(\tau_n'-\tau_{n-a_n}')}=\frac{We^{\lambda_A\tau_n'}}{A(\tau_n')}\frac{A(\tau_n')}{n}\frac{n}{n-a_n}\frac{n-a_n}{A(\tau_{n-a_n}')}\frac{A(\tau_{n-a_n}')}{We^{\lambda_A\tau_{n-a_n}'}},
\ea
we have that as $n\rightarrow\infty$
\ba
\tau_n'-\tau_{n-a_n}'\rightarrow0.
\ea
By Lemma \ref{suplemma},
\[B^{(n)}(\sigma'_n)\leq a_n\]
for sufficiently large $n$. For such $n$
\[A(\sigma_n')\geq n-a_n,\]
so
\[\sigma_n'\geq\tau_{n-a_n}',\]
and hence
\[
0\leq\tau_n'-\sigma_n'\leq\tau_n'-\tau_{n-a_n}'.
\]
\end{proof}
\begin{proof}[Proof of Theorem \ref{times}]Combine Lemmas \ref{easytimes} and \ref{tslemma} to see that conditioning on $W>0$,
\[\lim_{n\rightarrow\infty}K^{(n)}(\rho_n+t)= K^*(t)\]
almost surely, for each $t\in\mathbb{R}$. Convergence in finite dimensional distributions follows. Then apply Lemmas \ref{conditioning} and \ref{we} so that we may instead condition on the $\rho_n<\infty$.
\end{proof}
\begin{proof}[Proof of Propositions \ref{lpsml}, \ref{cs}, and \ref{lc}]
By Lemmas \ref{easytimes} and \ref{tslemma}, conditioning on $W>0$,
\[
\lim_{n\rightarrow\infty}\rho_n+t-T_i^{(n)}= t-T_i^*
\]
almost surely. Use that the $Y_i$ are almost surely continuous at  $t-T_i^*$. Then apply Lemmas \ref{conditioning} and \ref{we}, to condition on $\rho_n<\infty$.
\end{proof}

\begin{proof}[Proof of Proposition \ref{ltsmt}]
It needs to be shown that for each $t\in\mathbb{R}$
\[N\left(\int_{-t_n}^{t}\nu_nA(t_n+s)ds\right)\rightarrow N\left(\int_{-\infty}^tW\eta e^{\lambda_As}ds\right)\]
almost surely, as $n\rightarrow\infty$. Indeed, writing
\[\nu_nA(t_n+s)=\nu_ne^{\lambda_At_n}\frac{A(t_n+s)}{e^{\lambda_A(t_n+s)}}e^{\lambda_As},
\]
one sees that $\nu_nA(t_n+s)$ converges to the appropriate limit and is dominated by a multiple of $e^{\lambda_As}$.
\end{proof}
Proposition \ref{ltsml} follows by continuity.

\subsection*{Proofs for Section \ref{sec:finite}}
We first make note of a classic result, which can be found in \cite{ar}.
\begin{lemma}\label{oe}
Assume that $\beta_A=0$. For each $n$, $\left(\tau_n-\tau_k\right)_{k=1}^{n-1}$ has the same distribution as a collection of $n-1$ i.i.d. Exponential($\alpha_A$) random variables, which are ordered by size.

\end{lemma}
\begin{proof}[Proof of proposition \ref{finite}]
For each $i\in\mathbb{N}$ let $(T_{i,j})_{j\in\mathbb{N}}$ be the occurrence times of a homogeneous Poisson process on $[0,\infty)$ with intensity $\nu$. These are the mutation times corresponding to one particular wildtype cell present from time $0$. 
Noting that
\ba
A(t)=\sum_{i=1}^\infty 1_{[\tau_i,\infty)}(t),
\ea
it is apparent that the mutation times of all wildtype cells are distributed according to
\ba
(\tau_i+T_{i,j})_{i,j\in\mathbb{N}}.
\ea
The number of mutants at time $t$ is
\ba
B(t)\overset{d}{=}\sum_{i,j\in\mathbb{N}}1_{\{t-\tau_i-T_{i,j}\geq0\}}Y_{i,j}(t-\tau_i-T_{i,j})=\sum_{i\in\mathbb{N}}1_{\{t-\tau_i\geq0\}}D_i(t-\tau_i),
\ea
where
\ba
D_i(t)=\sum_{j\in\mathbb{N}} 1_{\{t-T_{i,j}\geq0\}}Y_{i,j}(t-T_{i,j})\overset{d}{=}\sum_{j=1}^{K_i(t)}Y_{i,j}(U_{i,j}t).
\ea
The $D_i(\cdot)$ are i.i.d. Now, using Lemma \ref{oe},
\ba
B(\tau_n)\overset{d}{=}\sum_{i=1}^{n-1}D_i(\tau_n-\tau_i)
\overset{d}{=}\sum_{i=1}^{n-1}D_i(\xi_i),
\ea
and by substituting $D_i(\cdot)$ the result is obtained.

\end{proof}

\begin{proof}[Proof of Lemma \ref{connect}]We will show that the events $\{B(\sigma_n)\leq k\}$ and $\{B(\tau_{n-k})\leq k\}$ are equal, using the monotonicity of $A(\cdot)$ and $B(\cdot)$ and the fact that $A(\sigma_n)+B(\sigma_n)=n$. First assume that $B(\sigma_n)\leq k$. Then $A(\sigma_n)\geq n-k$, so $\sigma_n\geq\tau_{n-k}$, and therefore $B(\tau_{n-k})\leq k$. Now assume that $B(\sigma_n)>k$. Then $A(\sigma_n)<n-k$, so $\sigma_n<\tau_{n-k}$, and hence $B(\tau_{n-k})>k$.

\end{proof}

\subsection*{Proofs for Section \ref{sec:sitefreq}}
\begin{proof}[Proof of Propositions \ref{sfsp} and \ref{sfst}] For $i\in\{1,..,S\}$ and $n\in\mathbb{N}$, $\left[C^{(n)}(\cdot)-B_i^{(n)}(\cdot)\right]$ is a birth-death branching process with birth and death rates $a(1-\mu_n)$ and $b(1-\mu_n)$. But we wish to make use of the proofs for the two-type model. Hence we will rescale time by a factor of $(1-\mu_n)$.

On a fresh probability space put a birth-death process, $(\tilde{A}(t))_{t\geq0}$, with birth and death rates $a$ and $b$. Put an independent Poisson process, $(\tilde{N}(t))_{t\geq0}$, with intensity $1$. And for $i,n\in\mathbb{N}$ put the birth-death processes $(\tilde{Y}_i^{(n)}(t))_{t\geq0}$, which we ask to satisfy:
\begin{enumerate}
\item
The $\tilde{Y}_i^{(n)}(\cdot)$ have birth and death rates $a/(1-\mu_n)$ and $b$, and initial condition $\tilde{Y}_i^{(n)}(0)=1$.
\item
For each $n$, the $\tilde{Y}_i^{(n)}(\cdot)$ are independent ranging over $i$.
\item
The $\tilde{Y}_i^{(n)}(\cdot)$ are independent of $\tilde{A}(\cdot)$ and $\tilde{N}(\cdot)$.
\item
For each $i$, $\lim_{n\rightarrow\infty}\tilde{Y}_i^{(n)}(\cdot)=\tilde{Y}_i(\cdot)$ exists almost surely, in the standard Skorokhod sense on the space of cadlag functions $\mathbb{D}[0,\infty)$.
\end{enumerate}
Define $\tilde{K}^{(n)}(t)=\tilde{N}\left(\frac{a\mu_n}{1-\mu_n}\int_0^t\tilde{A}(s)ds\right)$ and $\tilde{T}_i^{(n)}=\inf\{t\geq0:\tilde{K}^{(n)}(t)= i\}$. Define
\[\tilde{B}^{(n)}(t)=\sum_{i=1}^{\tilde{K}^{(n)}(t)}\tilde{Y}_i^{(n)}(t-\tilde{T}_i^{(n)}),
\]
and then $\tilde{\sigma}_n=\inf\{t\geq0:\tilde{A}(t)+\tilde{B}^{(n)}(t)\geq n\}$.

We have just defined a slight adaptation of the framework used for the small mutation limits of the two-type model. The proofs for the two-type model are readily adapted to this new setting. Here, the mutation rates are $a\mu_n/(1-\mu_n)$. Suppose that the $\mu_n$ satisfy the condition of Proposition \ref{sfsp}, then $\lim_{n\rightarrow\infty}na\mu_n/(1-\mu_n)=\theta$. Follow the proof of Proposition \ref{lpsml} to see that
\[
\lim_{n\rightarrow\infty}\mathbb{P}[\tilde{B}^{(n)}(\tilde{\sigma}_n)=k|\tilde{\sigma}_n<\infty]=\mathbb{P}[B^*=k].
\]
Then, use that
\[(B^{(n)}_i(t))_{t\geq0}\overset{d}{=}(\tilde{B}^{(n)}((1-\mu_n)t))_{t\geq0}
\]
and
\[\sigma_n'\overset{d}{=}\tilde{\sigma}_n/(1-\mu_n),
\]
to obtain
\[
\lim_{n\rightarrow\infty}\mathbb{P}[B_i^{(n)}(\sigma'_n)=k|\sigma'_n<\infty]=\mathbb{P}[B^*=k].
\]
This gives Proposition \ref{sfsp}.

Similarly, if the $\mu_n$ and $t_n$ satisfy the conditions of Proposition \ref{sfst}, then $\lim_{n\rightarrow\infty}e^{\lambda t_n/(1-\mu_n)}a\mu_n/(1-\mu_n)=\eta$. Follow the proof of Proposition \ref{ltsml} to see that
\[
\lim_{n\rightarrow\infty}\mathbb{P}[\tilde{B}^{(n)}(t_n/(1-\mu_n))=k]=\mathbb{P}[B^\circ=k].
\]
Then
\[
\lim_{n\rightarrow\infty}\mathbb{P}[B_i^{(n)}(t_n)=k]=\mathbb{P}[B^\circ=k],
\]
as required.
\end{proof}

\section*{Acknowledgements}
We thank Michael Nicholson and Stefano Avanzini for helpful discussions. We thank two anonymous referees for corrections and insightful comments. David Cheek was supported by The Maxwell Institute Graduate School in Analysis and its Applications, a Centre for Doctoral Training funded by the UK Engineering and Physical
Sciences Research Council (grant EP/L016508/01), the Scottish Funding Council, Heriot-Watt University and the University of Edinburgh.

\appendix
\section{Generalised two-type model}\label{sec:app}
Here we present a generalisation of Kendall's model in which the results of Sections \ref{sec:lpsm} and \ref{sec:ltsm} are valid. The broader framework encompasses more general branching processes as well as semi-deterministic versions of the model (for example \cite{lc,ka,dm}).

Consider the model defined in Section \ref{sec:mod}. Relax the requirement that $A(\cdot)$ and the $Y_i(\cdot)$ need to be birth-death branching processes. Instead let $A(\cdot)$ and the $Y_i(\cdot)$ be $[0,\infty)$-valued cadlag processes. Demand further that there exists $\lambda_A>0$ and a non-negative random variable $W$ with
\begin{enumerate}
\item
$\lim_{t\rightarrow\infty}e^{-\lambda_At}A(t)=W$, and
\item
$\{W=0\}=\{\exists T>0,\forall t\geq T,A(t)=0\}$,
\end{enumerate}
almost surely. We claim that Theorem \ref{times} and Propositions \ref{tausigma}, \ref{lpsml},  \ref{cs}, \ref{lc}, \ref{ltsmt}, and \ref{ltsml} remain valid.

To see this, only the proof of Lemma \ref{easytimes} requires additional work.  Observe that conditioning on $W>0$,

\begin{itemize}
\item
$\tau_n<\infty$ for each $n\in\mathbb{N}$, and
\item
$\lim_{n\rightarrow\infty}\tau_n=\infty$.
\end{itemize}
Then we need Lemma \ref{alo}.

\begin{lemma}\label{alo}
Conditioning on $W>0$,
\[
\lim_{n\rightarrow\infty}n^{-1}A(\tau_n)=1
\]
almost surely.
\begin{proof} Fix $\omega\in\{W>0\}$ and $\epsilon>0$. Then there exists $T>0$ such that for all $t>T$,
\[
|e^{-\lambda_A t}A(t)-W|\leq \epsilon.
\]
There is some $N\in\mathbb{N}$ such that for any integer $n\geq N$, $\tau_n>T$. Now, for all such $n$,
\[
|e^{-\lambda_A\tau_n}A(\tau_n^-)-W|=\lim_{t\uparrow \tau_n}|e^{-\lambda_A t}A(t)-W|\leq\epsilon,
\]
where $A(\tau_n^{-})=\lim_{t\uparrow \tau_n}A(t)$.
That is to say,
\[\lim_{n\rightarrow\infty}e^{-\lambda\tau_n}A(\tau_n^-)=W.
\]
Finally,
\[
1\leq\frac{A(\tau_n)}{n}\leq\frac{A(\tau_n)}{A(\tau_n^-)}=\frac{A(\tau_n)}{e^{\lambda_A\tau_n}}\frac{e^{\lambda_A\tau_n}}{A(\tau_n^-)}\rightarrow1.\color{black}
\]
\end{proof}
\end{lemma}

\bibliographystyle{imsart-nameyear}

\bibliography{References} 

\begin{thebibliography}{10}

\bibitem{wa}
W.~P. Angerer.
\newblock An explicit representation of the Luria-Delbr{\"u}ck distribution.
\newblock {\em Journal of Mathematical Biology}, 42(2):145--174, 2001.

\bibitem{ak}
T.~Antal and P.~L. Krapivsky.
\newblock Exact solution of a two-type branching process: Models of tumor
  progression.
\newblock {\em Journal of Statistical Mechanics: Theory and Experiment},
  P08018, 2011.

\bibitem{an}
K.~B. Arthreya and P.~Ney.
\newblock {\em Branching Processes}.
\newblock Dover Publications, 2004.

\bibitem{belife}
I.~Bozic et~al.
\newblock Evolutionary dynamics of cancer in response to targeted combination
  therapy.
\newblock {\em Elife}, (2):e00747, 2013.

\bibitem{ib}
I.~Bozic, J.~M. Gerold, and M.~A. Nowak.
\newblock Quantifying clonal and subclonal passenger mutations in cancer
  evolution.
\newblock {\em PLOS Computational Biology}, 12(2):e1004731, 2016.

\bibitem{djla}
L.~A. Diaz~Jr et~al.
\newblock The molecular evolution of acquired resistance to targeted EGFR
  blockade in colorectal cancers.
\newblock {\em Nature}, 486(7404):537, 2012.

\bibitem{ding}
D.~Dingli et~al.
\newblock The emergence of tumor metastases.
\newblock {\em Cancer Biology and Therapy}, 6(3):383--390, 2007.

\bibitem{durr}
R.~Durrett.
\newblock Population genetics of neutral mutations in exponentially growing
  cancer cell populations.
\newblock {\em The Annals of Applied Probability}, 23(1):230--250, 2013.

\bibitem{rdbp}
R.~Durrett.
\newblock {\em Branching Process Models of Cancer}.
\newblock Springer, 2014.

\bibitem{dflmm}
R.~Durrett et~al.
\newblock Evolutionary dynamics of tumor progression with random fitness
  values.
\newblock {\em Journal of Theoretical Population Biology}, 78(1):54--66, 2010.

\bibitem{dm}
R.~Durrett and S.~Moseley.
\newblock Evolution of resistance and progression to disease during clonal
  expansion of cancer.
\newblock {\em Theoretical Population Biology}, 77(1):42--48, 2010.

\bibitem{fs}
P.~Flajolet and R.~Sedgewick.
\newblock {\em Analytic Combinatorics}.
\newblock Cambridge University Press, 2009.

\bibitem{fl}
J.~Foo and K.~Leder.
\newblock Dynamics of cancer recurrence.
\newblock {\em The Annals of Applied Probability}, 23(4):1437--1468, 2013.

\bibitem{hm}
H.~Haeno and F.~Michor.
\newblock The evolution of tumor metastases.
\newblock {\em Journal of Theoretical Biology}, (263):30--44, 2010.

\bibitem{hy}
A.~Hamon and B.~Ycart.
\newblock Statistics for the Luria-Delbr{\"u}ck distribution.
\newblock {\em Electronic Journal of Statistics}, 6:1251--1272, 2012.

\bibitem{inm}
Y.~Iwasa, M.~A. Nowak, and F.~Michor.
\newblock Evolution of resistance during clonal expansion.
\newblock {\em Genetics}, 172(4):2557--2566, 2006.

\bibitem{sj}
S~Janson.
\newblock Functional limit theorems for multitype branching processes and
  generalized p{\'o}lya urns.
\newblock {\em Stochastic Processes and their Applications}, 110(2):177--245,
  2004.

\bibitem{sjt}
S~Janson.
\newblock Limit theorems for triangular urn schemes.
\newblock {\em Probability Theory and Related Fields}, 134(3):417--452, 2006.

\bibitem{jea}
S.~Jones et~al.
\newblock Comparative lesion sequencing provides insights into tumor evolution.
\newblock {\em Proceedings of the National Academy of Sciences of the United
  States of America}, 105(11):4283--4288, 2008.

\bibitem{ka}
P.~Keller and T.~Antal.
\newblock Mutant number distribution in an exponentially growing population.
\newblock {\em Journal of Statistical Mechanics: Theory and Experiment},
  P01011, 2015.

\bibitem{dk}
D.~G. Kendall.
\newblock Birth-and-death processes, and the theory of carcinogenesis.
\newblock {\em Biometrika}, 47(1-2):13--21, 1960.

\bibitem{kl1}
D.~A. Kessler and H.~Levine.
\newblock Large population solution of the stochastic Luria–Delbr{\"u}ck
  evolution model.
\newblock {\em Proceedings of the National Academy of Sciences of the United
  States of America}, 110(29):11628--11687, 2013.

\bibitem{kl}
D.~A. Kessler and H.~Levine.
\newblock Scaling solution in the large population limit of the general
  asymmetric stochastic Luria-Delbr{\"u}ck evolution process.
\newblock {\em Journal of Statistical Physics}, 158(4):783--805, 2015.

\bibitem{k}
N.~Komarova.
\newblock Stochastic modeling of drug resistance in cancer.
\newblock {\em Journal of Theoretical Biology}, 239(3):351--366, 2006.

\bibitem{kwb}
N.~L. Komarova, L.~Wu, and P~Baldi.
\newblock The fixed-size Luria–Delbr{\"u}ck model with a nonzero death rate.
\newblock {\em Mathematical Biosciences}, 210(1):253--290, 2007.

\bibitem{isr}
J.~Kuipers, K.~Jahn, B.~J. Raphael, and N.~Beerenwinkel.
\newblock Single-cell sequencing data reveal widespread recurrence and loss of
  mutational hits in the life histories of tumors.
\newblock {\em Genome research}, 2017.

\bibitem{lc}
D.~E. Lea and C.~A. Coulson.
\newblock The distribution of the numbers of mutants in bacterial populations.
\newblock {\em Journal of Genetics}, 49(3):264--285, 1949.

\bibitem{ld}
S.~E. Luria and M.~Delbr{\"u}ck.
\newblock Mutations of bacteria from virus sensitivity to virus resistance.
\newblock {\em Genetics}, 48(6):419--511, 1943.

\bibitem{mni}
F.~Michor, M.~A. Nowak, and Y.~Iwasa.
\newblock Stochastic dynamics of metastasis formation.
\newblock {\em Journal of Theoretical Biology}, 240(4):521--530, 2006.

\bibitem{na}
M.~D. Nicholson and T.~Antal.
\newblock Universal asymptotic clone size distribution for general population
  growth.
\newblock {\em Bulletin of Mathematical Biology}, 78(11):2243--2276, 2016.

\bibitem{oi}
H.~Ohtsuki and H.~Innan.
\newblock Forward and backward evolutionary processes and allele frequency
  spectrum in a cancer cell population.
\newblock {\em Theoretical Population Biology}, 117:43--50, 2017.

\bibitem{ar}
A~R{\'e}nyi.
\newblock On the theory of order statistics.
\newblock {\em Acta Mathematica Hungarica}, 4(3-4):191--231, 1953.

\bibitem{rf}
W.~A. Rosche and P.~L. Foster.
\newblock Determining mutation rates in bacterial populations.
\newblock {\em Methods}, 20(1):4--17, 2000.

\bibitem{gs}
M.~J. Williams, B.~Werner, C.~P. Barnes, T.~A. Graham, and A.~Sottoriva.
\newblock Identification of neutral tumor evolution across cancer types.
\newblock {\em Nature Genetics}, 48:238--244, 2016.

\bibitem{qz}
Q.~Zheng.
\newblock Progress of a half century in the study of the Luria-Delbr{\"u}ck
  distribution.
\newblock {\em Mathematical Biosciences}, 162(1-2):1--32, 1999.

\end{thebibliography}
\end{document}